\documentclass[oneside,reqno,10pt]{amsart}
\usepackage{graphicx,amsfonts,amssymb,amsmath,amsthm,url,comment}
\usepackage{tikz}
\usepackage{color}
\usepackage[normalem]{ulem}
\usepackage{pdfsync}

\usepackage{hyperref}

\theoremstyle{plain} 
\newtheorem{theorem}             {Theorem} 
\newtheorem{lemma}      [theorem]{Lemma}
\newtheorem{corollary}  [theorem]{Corollary}

\theoremstyle{definition}

\theoremstyle{remark}
\newtheorem{remark}              {Remark}

\renewcommand{\geq}{\geqslant}
\renewcommand{\leq}{\leqslant}

\def\diag{\operatorname{diag}}

\def\trace{\operatorname{trace}}

\def\trace{\operatorname{trace}}

\def\End{\operatorname{End}}

\def\End{\operatorname{End}}

\def\vol{\operatorname{vol}}

\def\GL{\operatorname{GL}}

\def\SL{\operatorname{SL}}

\def\trace{\operatorname{trace}}

\def\End{\operatorname{End}}

\def\End{\operatorname{End}}

\def\vol{\operatorname{vol}}

\def\GL{\operatorname{GL}}

\def\SL{\operatorname{SL}}

\begin{document}

\author{Paul D. Nelson}
\address{ETH Zurich, Department of Mathematics, R{\"a}mistrasse 101, CH-8092, Zurich, Switzerland}
\email{paul.nelson@math.ethz.ch}

\title{Analytic isolation of newforms of given level}

\maketitle

\begin{abstract}
  We describe a method for understanding averages over newforms
  on $\Gamma_0(q)$ in terms of averages over all forms of some
  level.  
  The method is simplest when $q$ is divisible by the cubes of its prime divisors.
\end{abstract}

\section{Introduction\label{sec:intro}}
\label{sec-1}
Fix a positive integer $k$.
For each positive integer $q$,
let $\mathcal{A}(q)$ denote the space of weight $k$
holomorphic cusp forms on $\Gamma_0(q)$.  It is a
finite-dimensional inner product space.  Let
$\mathcal{A}^*(q) \leq \mathcal{A}(q)$ denote the
Atkin--Lehner \emph{newspace}
\cite{MR0268123}; it is the orthogonal complement of the
\emph{oldspace}, which is in turn the span of the forms
\[
\varphi|_d(z) := d^{k/2} \varphi(d z)
\]
taken over all proper
divisors $\ell \neq q$ of $q$, all divisors $d$ of $q/\ell$, and all
$\varphi \in \mathcal{A}(\ell)$.

The newspace  $\mathcal{A}^*(q)$ is
of
fundamental interest and importance 
because
of its strong
interaction
with the theory of Hecke operators.
In analytic
number theory, it is of particular interest to understand averages
(of Fourier coefficients, $L$-values, ...)  over the newspace.  Unfortunately, the basic tools for
studying such averages (e.g., trace formulas) apply most
directly to the larger spaces $\mathcal{A}(q)$.

There arises
the problem of relating averages over $\mathcal{A}^*(q)$ to
those over $\mathcal{A}(q)$.  Several authors\footnote{
  \label{footnote:1}see for instance \cite[\S3]{MR1369394},
  \cite[\S2]{MR1828743}, \cite[\S3]{MR2755090},
  \cite[\S5]{MR3334233}, \cite{2016arXiv160403224B}} have addressed
this problem via the Atkin--Lehner decomposition
\begin{equation}\label{eqn:atkin-lehner-decomp}
  \mathcal{A}(q) = \bigoplus_{\ell \mid q} \bigoplus_{d \mid
    \frac{q}{\ell}} \{\varphi|_d : \varphi \in
  \mathcal{A}^*(\ell) \}
  \quad \text{(\emph{non-orthogonal}
    direct sum)}
\end{equation}
followed by a computationally-involved Gram--Schmidt
orthogonalization.
In this article, we introduce an approach
which is more direct for certain values of $q$.
\begin{theorem}\label{cor:concrete}
  Suppose $q$ is divisible by the cube of every prime that divides it.
  Then for $z_1, z_2$ in the upper half-plane,
  \begin{equation}\label{eqn:newform-sum-bilinear}
    \sum_{\varphi \in \mathcal{B}^*(q)}
    \varphi(z_1) \overline{\varphi}(z_2)
    =
    \sum_{d, e \mid q}
    \mu(d) \mu(e)
    \sum_{\varphi \in \mathcal{B}(\frac{q}{d e})}
    \varphi|_d(z_1) \overline{\varphi|_d}(z_2),
  \end{equation}
  where $\mathcal{B}(\ell), \mathcal{B}^*(\ell)$ denote arbitrary
  orthonormal bases for $\mathcal{A}(\ell), \mathcal{A}^*(\ell)$
  defined using Petersson inner products with respect to normalized
  hyperbolic measures of volume independent of $\ell$,
  such as $\vol(\Gamma_0(\ell) \backslash \mathbb{H})^{-1} \, \frac{d x \, d y}{y^2}$, so that $\varphi \mapsto \varphi|_{d}$ is unitary.
\end{theorem}
The formulation of Theorem \ref{cor:concrete} requires only the definition of the
newspace as introduced by Atkin--Lehner \cite{MR0268123} in
1970, but the simple identity \eqref{eqn:newform-sum-bilinear}
does not appear to have been anticipated prior to this work
despite the considerable technical effort
expended
on the problems that it addresses.
The proof is short, simple and independent of
\eqref{eqn:atkin-lehner-decomp}.
The restriction on $q$ will be discussed in due course, and is
irrelevant
for our motivating applications in which it is a large power of a fixed prime.

Theorem \ref{cor:concrete} applies to problems in which one
\emph{knows} how to average over all forms of given level, but
\emph{wants} to average only over the newforms.  For the sake of
illustration, we record here three such applications.  Some of
the consequences to follow are new, some old. Theorem
\ref{cor:concrete} itself is new.

As we explain in detail in the body of the paper (see
\S\ref{sec:basic-idea}, \S\ref{sec:main}), Theorem
\ref{cor:concrete} may be understood as an identity of hermitian
forms (equivalently, self-adjoint operators) on the space of
modular forms (compare with \cite[Lem 4.1]{MR3325425}, for
instance).  The proofs of the applications to follow may then be
understood as the result of taking the Hilbert--Schmidt pairing of
that identity against a given sesquilinear form.

We begin with an overly simple application which may
nevertheless aid orientation.  Let $n$ be a natural number
coprime to $q$.  By applying the Hecke operator $T_n$ to either
variable and integrating over the diagonal $z_1 = z_2$ of
\eqref{eqn:newform-sum-bilinear}, we obtain a relation between
traces of Hecke operators acting on the newspace and on all
forms of given level:
\begin{equation}\label{eq:traces-newforms}
  \trace(T_n | \mathcal{A}^*(q))
  = \sum_{d,e|q} \mu(d) \mu(e)
  \trace(T_n | \mathcal{A}(\tfrac{q}{d e})).
\end{equation}
Applying the Eichler--Selberg trace formula to
each summand on the RHS of
\eqref{eq:traces-newforms}
gives a \emph{trace formula for newforms},
the case $n=1$ of which is a \emph{dimension formula for newspaces}.
Such formulas are not new: they follow
(for general $q$)
from
\eqref{eqn:atkin-lehner-decomp} and M{\"o}bius inversion (see
\cite[\S5.1]{MR1396897} or \cite[\S2]{MR1604056} and
\cite{MR2141534}).
The proof indicated here
does not use \eqref{eqn:atkin-lehner-decomp}.

Second, write the Fourier expansion of $\varphi \in \mathcal{A}(q)$
as
\[\varphi(z) = \sum_{n \geq 1} n^{k/2} \rho(n;\varphi) e^{2
  \pi i n z},\]
so that $\rho(n;{\varphi|_{d}}) = 1_{d \mid
  n} \rho(\tfrac{n}{d};{\varphi})$.
Let $m,n$ be positive integers.
By taking the $m$th (resp. $n$th) Fourier coefficient in the $z_1$
(resp. $z_2$) variable
of \eqref{eqn:newform-sum-bilinear},
we obtain
an identity
\begin{equation}\label{eqn:petersson}
  \Delta^*(m,n;q)
  =
  \sum_{\substack{
      d \mid \gcd(m,n,q)  \\
      e \mid q
    }
  }
  \mu(d) \mu(e)
  \Delta(\tfrac{m}{d},\tfrac{n}{d}; {\tfrac{q}{d e}}),
\end{equation}
expressing the averages of Fourier coefficients
\[
\Delta^*(m,n;q)
:= \sum_{\varphi \in \mathcal{B}^*(q)}
\rho(m;\varphi) \overline{\rho(n;\varphi)}
\]
over newforms
in terms of the analogous averages $\Delta(m,n;q)$
over $\mathcal{B}(q)$.
Applying the classical Petersson formula  to
the RHS of \eqref{eqn:petersson} gives a \emph{Petersson formula
  for newforms}.

The special case of \eqref{eqn:petersson} in
which the variables $m,n$ satisfy the coprimality constraint
$(m n, q) = 1$ and $q$ is a prime power was established by
D. Rouymi \cite[Prop. 9, Rmk. 4]{MR2755090} in his work on
newforms of level $p^\nu, \nu \rightarrow \infty$ after some
involved calculations along the lines indicated following
\eqref{eqn:atkin-lehner-decomp} (see also \cite[\S5]{MR3334233},
\cite{2016arXiv160403224B}).  The proof given here by way of
Theorem \ref{cor:concrete}
is simpler in
that it avoids explicit orthogonalization of the decomposition
\eqref{eqn:atkin-lehner-decomp}.
The uniformity of
\eqref{eqn:petersson}
with respect to the variables
$m,n$
appears to be relevant
for applications
such as those pursued recently in \cite{2016arXiv160806854P}.

Formulas of the shape \eqref{eqn:petersson} have already seen
diverse applications.  The references in footnote
\ref{footnote:1} contain several such applications, as well as
proofs of special cases of \eqref{eqn:petersson}.  Conversely,
by Fourier inversion, one can recover Theorem \ref{cor:concrete}
from the general case of \eqref{eqn:petersson}.



Third, let $\Psi : \Gamma_0(q)
\backslash \mathbb{H} \rightarrow \mathbb{C}$
be a measurable function of moderate growth,
and let $n$ be a natural number coprime to $q$.
The Petersson inner products
$\langle \varphi, \Psi \varphi \rangle$
are of basic interest in many questions (quantum unique
ergodicity, subconvexity, ...).
A formula relating their Hecke--twisted
first moments over newforms and over all forms
of given level
follows from the proof of
\eqref{eq:traces-newforms}
by weighting the integrand
by $\Psi(z)$ in the final step:
\begin{equation}\label{eq:twisted-diag-integral}
  \sum_{\varphi \in \mathcal{B}^* (q)}
  \left\langle \varphi, \Psi \cdot T_n \varphi \right\rangle
  =
  \sum_{d,e \mid q}
  \mu(d) \mu(e)
  \sum_{\varphi \in \mathcal{B} (\tfrac{q}{d e})}
  \left\langle \varphi|_d, \Psi \cdot T_n \varphi|_d \right\rangle.
\end{equation}
The summands on the RHS of \eqref{eq:twisted-diag-integral} may be studied by integrating $\Psi$
against the holomorphic kernel for $T_n$ on conjugates of
$\Gamma_0(q)$.  By the multiplicity one theorem, the function of
$n$ given by the  LHS of \eqref{eq:twisted-diag-integral} determines the family of inner products
$\langle \varphi, \Psi \varphi \rangle$ arising as $\varphi$
traverses an orthonormal basis of Hecke newforms.
The original
motivation for the present work is that
\eqref{eq:twisted-diag-integral} and its variants constitute
the first step in a method to study the quantum variance of
newforms of large level (see
\cite[\S7.1]{nelson-variance-73-2}).
Formula \eqref{eq:twisted-diag-integral} is new in all cases.





For many problems involving modular forms, the case of
\emph{squarefree} (or even \emph{prime}) level is often the
simplest and hence the natural first case to consider.  Our
method, perhaps counterintuitively, applies most directly to
\emph{cubefull} levels.  It applies also to levels that are not
necessarily cubefull, but becomes more complicated to implement
and is not clearly superior to existing approaches.
The present generality suffices for the depth
aspect in which levels are powers of a fixed prime and hence for
our motivating applications \cite{nelson-variance-73-2}.

The cubefull levels often exhibit representative phenomena.
Because of its directness and simplicity, our method may be
useful also for problems involving non-cubefull levels as a
first step towards understanding the expected truth.  This is
analogous (in several respects) to studying the asymptotics of
smoothly weighted sums $\sum_{n} f(n) W(n/x)$ before those of
their sharply-truncated counterparts $\sum_{n \leq x} f(n)$.


To explain the basic ideas behind the proof with minimal notation/prerequisites,
we record in \S\ref{sec:basic-idea} a direct proof of a
representative special case of
Theorem \ref{cor:concrete}.
We then formulate
in \S\ref{sec:main} our main result,
which may be understood as a local
representation-theoretic form of Theorem \ref{cor:concrete} that
applies also to Maass forms, on quotients attached to quaternion
algebras, over number fields, and (with minor
modifications) in half-integral weight;
it consists of
constructing an element of the Hecke algebra of $\GL_2$ over a
non-archimedean local field that projects onto the newvectors of
given log-conductor $\geq 3$.

To elucidate
that result from as many perspectives as possible,
we then give three short proofs.
Each relies on
a novel operator calculus for idempotents in the Hecke algebra
(\S\ref{sec-6}) which we verify
\begin{enumerate}
\item group-theoretically
(\S\ref{sec-7}), 
\item by reduction to a probabilistic assertion
concerning random non-backtracking walks on the Bruhat--Tits
tree (\S\ref{sec-8}), and 
\item using the Kirillov model and
recurrence relations for Hecke eigenvalues (\S\ref{sec:9}).
\end{enumerate}

In closing, we note that it would be natural and interesting
to extend the present work to the setting of
newvectors on $\GL_N$ \cite{MR620708}.

\section{The proof in a basic but representative case\label{sec:basic-idea}}
\label{sec-2}
We now prove Theorem \ref{cor:concrete} in the prime-cubed case
$q = p^3$, which already captures the key ideas.
The general case will be deduced in \S\ref{sec-5} from
our main local results.

The required identity \eqref{eqn:newform-sum-bilinear} specializes to
\begin{equation}\label{eqn:specialized-goal} 
  \begin{split}
    \sum_{\varphi \in \mathcal{B}^*(p^3)}
    \varphi(z_1) \overline{\varphi}(z_2)
    &=
    \sum_{\varphi \in \mathcal{B}(p^3)}
    \varphi(z_1)
    \overline{\varphi}(z_2) 
    -
    \sum_{\varphi \in \mathcal{B}(p^2)}
    \varphi(z_1)
    \overline{\varphi}(z_2) \\
    &\quad -
    \sum_{\varphi \in \mathcal{B}(p^2)}
    \varphi|_p(z_1)
    \overline{\varphi|_p}(z_2) 
    +
    \sum_{\varphi \in \mathcal{B}(p)}
    \varphi|_p(z_1)
    \overline{\varphi|_p}(z_2).
  \end{split}
\end{equation}
For $i, j \in \{0,1,2,3\}$ with $i \leq j$,
let
$E_{ij} : \mathcal{A}(p^3) \rightarrow \mathcal{A}(p^3)$
denote the orthogonal projector
onto the subspace
$\mathcal{A}_{i j}
:=
\{\varphi|_{p^i} : \varphi \in \mathcal{A}(p^{j-i})\}$.
One has
$\mathcal{A}_{i j} = \{\varphi \in \mathcal{A}(p^3) :
\varphi|\gamma = \varphi \text{ for all } \gamma \in \Gamma_{i
  j}\}$
where
$\varphi|\gamma$ is the slash operator used
to define
the automorphy of $\varphi$
and $\Gamma_{i j}$ is the group
\[
\Gamma_{i j} :=
\begin{bmatrix}
  p^{-i} &  \\
  & 1
\end{bmatrix}
\Gamma_0(p^{j-i})
\begin{bmatrix}
  p^i &  \\
  & 1
\end{bmatrix}
=
\begin{bmatrix}
  \mathbb{Z}  & p^{-i} \mathbb{Z}  \\
  p^j \mathbb{Z} &  \mathbb{Z} 
\end{bmatrix} \cap \SL_2(\mathbb{Q})
\]
fitting into the lattice diagram
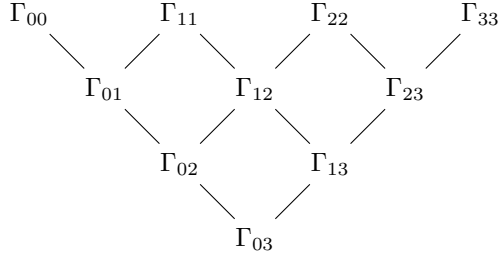
\[
\begin{tikzpicture}
  \node (00) at (-3,2) {$\Gamma_{0 0}$};
  \node (11) at (-1,2) {$\Gamma_{1 1}$};
  \node (22) at (1,2) {$\Gamma_{2 2}$};
  \node (33) at (3,2) {$\Gamma_{3 3}$};
  \node (01) at (-2,1) {$\Gamma_{0 1}$};
  \node (12) at (0,1) {$\Gamma_{1 2}$};
  \node (23) at (2,1) {$\Gamma_{2 3}$};
  \node (02) at (-1,0) {$\Gamma_{0 2}$};
  \node (13) at (1,0) {$\Gamma_{1 3}$};
  \node (03) at (0,-1) {$\Gamma_{0 3}$};
  \draw (03) -- (02) -- (01) -- (00)
  (03) -- (13) -- (23) -- (33)
  (02) -- (12) -- (11)
  (01) -- (11)
  (12) -- (22)
  (13) -- (12)
  (23) -- (22);
\end{tikzpicture}
\]
with the smallest group $\Gamma_{0 3} = \Gamma_0(p^3)$ at the
bottom, the largest groups (all conjugates of $\Gamma_{0 0} = \SL_2(\mathbb{Z})$) along the top,
the chain $\Gamma_{0 j} = \Gamma_0(p^j)$ along the left edge,
and with $\Gamma_{i j} \cap \Gamma_{j k} = \Gamma_{i k}$
for $0 \leq i \leq j \leq k \leq 3$.
The projector $E_{i j}$ may be expressed concretely as the averaging operator
\[
E_{i j} \varphi
=
\frac{1}{|\Gamma_{03} \backslash \Gamma_{ij}|}
\sum_{\gamma \in \Gamma_{03}
  \backslash \Gamma_{ij}}
\varphi|\gamma.
\]
Because
$\{\varphi|_{p^i} : \varphi \in \mathcal{B}(p^{j-i})\}$
extends to an orthonormal basis
of $\mathcal{A}(p^3)$,
one has
\[\sum_{\varphi \in \mathcal{B}(p^{j-i})}
\varphi|_{p^i}(z_1) \overline{\varphi|_{p^i}}(z_2)= 
\sum_{\varphi \in \mathcal{B}(p^3)}
E_{ij} \varphi(z_1) \overline{\varphi}(z_2),\]
so our specialized goal \eqref{eqn:specialized-goal}
may be rewritten as
\[
\begin{split}
  \sum_{\varphi \in \mathcal{B}^*(p^3)}
  \varphi(z_1) \overline{\varphi}(z_2)
  &=
  \sum_{\varphi \in \mathcal{B}(p^3)}
  E_{0 3} \varphi(z_1)
  \overline{\varphi}(z_2) 
  -
  \sum_{\varphi \in \mathcal{B}(p^3)}
  E_{0 2} \varphi(z_1)
  \overline{\varphi}(z_2)  \\
  &\quad
  -
  \sum_{\varphi \in \mathcal{B}(p^3)}
  E_{1 3} \varphi(z_1)
  \overline{\varphi}(z_2) 
  +
  \sum_{\varphi \in \mathcal{B}(p^3)}
  E_{1 2} \varphi(z_1)
  \overline{\varphi}(z_2),
\end{split}
\]
or equivalently,
with the definition
$E_{03}^* := E_{03} - E_{02} - E_{13} + E_{12}$,
as
\[
\sum_{\varphi \in \mathcal{B}^*(p^3)}
\varphi(z_1) \overline{\varphi}(z_2)
=
\sum_{\varphi \in \mathcal{B}(p^3)}
E_{03}^*
\varphi(z_1)
\overline{\varphi}(z_2).
\]
In other words,
we must show the following:
\begin{lemma}\label{prop:level-p-cubed-projector}
  $E_{03}^*$ defines the orthogonal projector onto
  the newspace $\mathcal{A}^*(p^3)$.
\end{lemma}
\begin{remark}
  The conclusion of Lemma \ref{prop:level-p-cubed-projector}
  is not altogether formal.
  For instance, it fails if one replaces ``$3$'' by ``$2$.''
\end{remark}
The operators $E_{02}, E_{13}$ and $E_{12}$ are self-adjoint
idempotents with image in the oldspace
and hence kernel containing the newspace,
while
$E_{03}$ is the identity,
so
$E_{03}^*$ restricts to the identity on newspace.
To conclude that $E_{03}^*$
orthogonally projects
onto the newspace,
it remains only to verify that it
annihilates the oldspace;
as the latter is spanned by the images
of $E_{02}$ and $E_{13}$,
it suffices to show that
$E_{03}^* \circ E_{02} = 0$ and
$E_{03}^* \circ E_{13} = 0$.
We verify here the first of these identities, the proof of the
second being similar.
We claim that
\begin{equation}\label{eq:comp-1}
  E_{03} \circ E_{02} = E_{02},
  \quad
  E_{02} \circ E_{02} = E_{02}, \quad     E_{12} \circ E_{02} = E_{12},
\end{equation}
and
\begin{equation}\label{eq:comp-3}
  E_{13} \circ E_{02} = E_{12}
\end{equation}
from which
it follows that
$E_{03}^* \circ E_{02}
= E_{02} - E_{02} - E_{12} + E_{12} = 0$,
as required.
The identities
\eqref{eq:comp-1}
are consequences of the transitivity of orthogonal
projection
onto nested subspaces.
The interesting identity is thus \eqref{eq:comp-3},
which we may write thanks to
the third identity in \eqref{eq:comp-1}
in the equivalent form
$E_{1 3} \circ E_{0 2} = E_{1 2} \circ E_{0 2}$ and then in
terms
of averaging operators
as the assertion that
for all $\varphi \in \mathcal{A}_{0 2}$,
\[
\frac{1}{|\Gamma_{0 3} \backslash \Gamma_{1 3}|}
\sum_{\gamma \in \Gamma_{0 3} \backslash \Gamma_{1 3}}
\varphi | \gamma
=
\frac{1}{|\Gamma_{0 2} \backslash \Gamma_{1 2}|}
\sum_{\gamma \in \Gamma_{0 2} \backslash \Gamma_{1 2}}
\varphi | \gamma.
\]
To that end, it suffices to verify that the natural
map of coset spaces
$\Gamma_{0 3} \backslash \Gamma_{1 3} \rightarrow \Gamma_{0 2}
\backslash \Gamma_{ 1 2}$
induced by the inclusions $\Gamma_{1 3} \leq \Gamma_{1 2},
\Gamma_{0 3} \leq \Gamma_{0 2}$ is bijective.
The injectivity follows from the evident identity
$\Gamma_{1 3} \cap \Gamma_{0 2} = \Gamma_{0 3}$,
while the surjectivity, which is the crux of the whole matter,
is
given as follows:
\begin{lemma}\label{lem:basic-surjective}
  The map $\Gamma_{0 2} \times \Gamma_{1 3} \ni (\gamma_1,\gamma_2) \mapsto \gamma_1 \gamma_2 \in \Gamma_{1 2}$ is surjective.
\end{lemma}
\begin{proof}
  We must show that every $\gamma \in \SL_2(\mathbb{Q})$ satisfying
  \[
  \gamma \in
  \begin{bmatrix}
    \mathbb{Z}  & p^{-1} \mathbb{Z}  \\
    p^2 \mathbb{Z}  & \mathbb{Z} 
  \end{bmatrix}
  \]
  arises as the product $\gamma = \gamma_1 \gamma_2$
  of two $\gamma_1, \gamma_2 \in \SL_2(\mathbb{Q})$
  satisfying
  \[
  \gamma_1
  \in
  \begin{bmatrix}
    \mathbb{Z}  & \mathbb{Z}  \\
    p^2 \mathbb{Z}  & \mathbb{Z} 
  \end{bmatrix},
  \quad 
  \gamma_2 \in
  \begin{bmatrix}
    \mathbb{Z}  & p^{-1} \mathbb{Z}  \\
    p^3 \mathbb{Z}  & \mathbb{Z} 
  \end{bmatrix}.
  \]
  To simplify calculations, we conjugate by
  $\left(
    \begin{smallmatrix}
      p&\\
      &1
    \end{smallmatrix}
  \right)$ and reduce modulo $p^3$.
  Set $\mathfrak{o} := \mathbb{Z}/p^3$,
  $\mathfrak{p} := p \mathbb{Z} / p^3 < \mathfrak{o}$.
  By the surjectivity of the natural map $\SL_2(\mathbb{Z})
  \rightarrow \SL_2(\mathfrak{o})$,
  we reduce to verifying
  that every
  \[
  \gamma 
  = \begin{pmatrix}
    a & b \\
    c & d
  \end{pmatrix}
  \in
  \begin{bmatrix}
    \mathfrak{o}  & \mathfrak{o}  \\
    \mathfrak{p}   & \mathfrak{o} 
  \end{bmatrix} \cap \SL_2(\mathfrak{o})
  \]
  arises as the product $\gamma = \gamma_1 \gamma_2$
  of some
  \[
  \gamma_1
  \in
  \begin{bmatrix}
    \mathfrak{o}   & \mathfrak{p}   \\
    \mathfrak{p}   & \mathfrak{o} 
  \end{bmatrix} \cap \SL_2(\mathfrak{o}),
  \quad 
  \gamma_2 \in
  \begin{bmatrix}
    \mathfrak{o}   & \mathfrak{o}  \\
    \mathfrak{p}^2  & \mathfrak{o}
  \end{bmatrix} \cap \SL_2(\mathfrak{o}).
  \]
  To that end,
  we note that $\det(\gamma) = 1, c \in \mathfrak{p}$ implies $a \in
  \mathfrak{o}^\times$
  and take
  \[
  \gamma_1 :=
  \begin{pmatrix}
    1 & 0 \\
    c/a & 1
  \end{pmatrix},
  \quad
  \gamma_2 :=
  \begin{pmatrix}
    a & b \\
    0 & d - b c/a
  \end{pmatrix}.
  \]
  This completes the proof of Lemma
  \ref{lem:basic-surjective},
  hence of Lemma \ref{prop:level-p-cubed-projector},
  hence of the $q = p^3$ case of Theorem \ref{cor:concrete}.
\end{proof}

\begin{remark}
  The ``standard'' approach to proving something like Theorem
  \ref{cor:concrete} (see e.g. \cite{MR2755090} or 
  \S\ref{sec:9}) would be to check it one Hecke-irreducible
  subspace at a time using explicit formulas
  derived from Atkin--Lehner theory.  The present
  observation is that it is in some cases more efficient
  to work with the congruence subgroups themselves.
\end{remark}
\section{Statement of the main local result\label{sec:main}}
\label{sec-3}
We now formulate the main result of this article,
which as indicated earlier may be understood as a local,
flexibly applicable
form of Theorem \ref{cor:concrete}
(see \cite[\S7.1]{nelson-variance-73-2}
for an application).

Let $k$ be a non-archimedean local field with ring of integers
$\mathfrak{o}$, maximal ideal $\mathfrak{p}$,
uniformizer $\varpi$, and $q := \#\mathfrak{o}/\mathfrak{p}$.
Let $G$ be a closed subgroup of $\GL_2(k)$ that contains $\SL_2(k)$. Equip $G$ with some Haar measure $d g$.
By a \emph{segment}, we shall mean a nonempty
finite consecutive set of integers,
denoted $m..n := \{m, m+1, \dotsc, n\}$ for some
integers $m,n$ with $m \leq n$.
The cardinality of a segment $\ell$ is
denoted $\# \ell$, thus $\# m..n := |m-n| + 1$.
For each segment $\ell = m..n$,
denote by
\[
R_{\ell} :=
\begin{bmatrix}
  \mathfrak{o}  & \mathfrak{p}^{-m} \\
  \mathfrak{p}^{n} & \mathfrak{o} 
\end{bmatrix}
\]
the Eichler order
of level $|m-n| = \# \ell - 1$
indexed by $\ell$, regarded as a
geodesic segment on the Bruhat--Tits tree (see
e.g. \cite{MR580949} or \cite{MR1954121} or
\cite[\S1.2]{MR2729264}).
Let $R_\ell^\times$ denote its
unit group and $K_{\ell} := R_\ell^\times \cap G < G$ the
intersection of that unit group with $G$.
Denote by $1_{K_{\ell}}$ the
characteristic function.

Let $\pi$ be a smooth representation of $G$.
For $f \in C_c^\infty(G)$,
denote by $\pi(f) \in \End(\pi)$
the operator
$\pi(f) v := \int_{g \in G} f(g)  \pi(g) v \, d g$.
For each segment $\ell$,
denote by $\pi[\ell] := \pi^{K_{\ell}}$ the subspace of vectors
fixed
by $K_\ell$.
The \emph{standard projector} onto
$\pi[\ell]$
is the averaging operator
$e_{\ell} := \pi(\vol(K_{\ell})^{-1} 1_{K_\ell}) \in \End(\pi)$;
it is an idempotent projector with image $\pi[\ell]$.
The \emph{standard complement} of $\pi[\ell]$ is the kernel
of $e_\ell$.
One has
\begin{equation}\label{eq:pi-ell-decomp}
  \pi = \pi[\ell] \oplus \ker(e_\ell).
\end{equation}
Note that $\ell \supseteq \ell'$ implies $\pi[\ell] \supseteq
\pi[\ell']$.
Set
$\pi[\ell]^{\flat} := \sum_{\ell' \subsetneq \ell}
\pi[\ell']$, with the sum taken over all proper subsegments
$\ell'$ of $\ell$.
The \emph{standard complement}
of $\pi[\ell]^{\flat}$ in $\pi[\ell]$
is
\[
\pi[\ell]^{\sharp} := \{v \in \pi[\ell] :
e_{\ell'} v = 0 \text{ for
  all } \ell' \subsetneq \ell\}.
\]
One has
\begin{equation}\label{eq:pi-ell-sharp-decomp}
  \pi[\ell] = \pi[\ell]^{\sharp}
\oplus \pi[\ell]^{\flat}.
\end{equation}
The \emph{standard projector} onto
$\pi[\ell]^{\sharp}$ is the projector
$\pi \rightarrow \pi[\ell]^{\sharp}$
afforded by the decompositions
\eqref{eq:pi-ell-decomp}
and \eqref{eq:pi-ell-sharp-decomp}.

If $\pi$ is unitary, then ``standard projector''
and ``standard complement'' have the same meanings as
``orthogonal projector'' and ``orthogonal complement''.
Indeed,
in that case the invariance under inversion of
the Haar measures on the compact groups $K_{\ell}$
implies that
the projectors $e_\ell$ and
$e_{\ell}^*$ are self-adjoint, hence define
orthogonal projections (since a projection is
orthogonal
if and only if it is self-adjoint).

\begin{theorem}\label{thm:main}
  Let $\ell = m..n$ be a segment with $\# \ell - 1 = |m-n| \geq 3$.
  Then
  $e_{\ell}^* := e_{m..n} - e_{m+1..n} - e_{m..n-1} +
  e_{m+1..n-1}$ is
  the standard
  projector onto $\pi[\ell]^{\sharp}$.
\end{theorem}

\begin{remark}\label{rem:reduce-to-sl2}
  The general case of Theorem \ref{thm:main}
  reduces to the case $G = \SL_2(k)$
  by convolving against the characteristic function of
  $\begin{bmatrix}
    \det(G) &  \\
     & 1
  \end{bmatrix}$.
\end{remark}

\begin{remark}
  We have assumed in Theorem \ref{thm:main}
  neither that $\pi$ is irreducible nor generic,
  hence this result
  lies somewhat shallower
  than the fundamental results of local newvector theory \cite{MR0337789, Sch02}.
  For instance,
  in the case $G = \GL_2$, it does not depend upon the existence of the Kirillov model.
\end{remark}
\section{Interpretation for generic representations of $\GL_2$}
\label{sec-4}
We record here what Theorem \ref{thm:main} says when $G = \GL_2(k)$ and $\pi$ is a generic irreducible representation with unramified central character.
In that case,
local newvector theory \cite{MR0337789, Sch02} says that one may attach to $\pi$ a
nonnegative integer $c(\pi)$,
its \emph{log-conductor},
with the property that
$\pi[0..n] \neq 0$
if and only if $n \geq c(\pi)$;
one knows then moreover that
$\dim \pi[0..n] = \max(0,n + 1 - c(\pi))$
and that
$\pi[0..n]^\sharp = 0$ unless $n = c(\pi)$,
in which case
$\dim \pi[0..n]^\sharp = 1$.
Since the subgroups $K_\ell, K_{\ell'}$ are conjugate
whenever $\# \ell = \# \ell '$, it follows
more generally
for any segment $\ell$
that
$\dim \pi[\ell] = \max(0,\# \ell - c(\pi))$
and
\begin{equation}\label{eqn:dimension-formulas-for-gl2-fixed-vectors}
\dim \pi[\ell]^\sharp = 
\begin{cases}
  1 & \text{ if } \# \ell - 1 = c(\pi),
  \\
  0 & \text{ otherwise.}
\end{cases}
\end{equation}
Theorem \ref{thm:main} thus implies the following:
\begin{corollary}\label{cor:for-generics}
  For $G = \GL_2(k)$, $\pi$ as above, and $\ell$ satisfying the assumptions of Theorem \ref{thm:main},
  one has $e_\ell^* = 0$ unless $c(\pi) = \# \ell - 1$, in which
  case
  $e_{\ell}^*$ is the standard projector onto the
  one-dimensional space
  $\pi[\ell] = \pi[\ell]^\sharp$.
\end{corollary}
\begin{remark}
  If $\pi$ is irreducible and non-generic, then it is
  one-dimensional,
  and $e_{\ell}^* = 0$ under the assumptions of Theorem \ref{thm:main}.
\end{remark}
\begin{remark} The formulation and proof of Corollary \ref{cor:for-generics}
  extend with the usual modifications to representations having ramified central
  character:
Let $\pi$ be a generic irreducible representation
of $G = \GL_2(k)$
with central character $\omega : k^\times \rightarrow
\mathbb{C}^\times$.
Denote by $\chi : \mathfrak{o}^\times \rightarrow
\mathbb{C}^\times$
the restriction of $\omega$.
Let
$k := c(\chi)$
denote the log-conductor of $\chi$,
i.e., the smallest $k \in \mathbb{Z}_{\geq 0}$ for which
$\chi$ restricts trivially to $\mathfrak{o}^\times \cap 1 + \mathfrak{p}^k$.
Assume that $k \geq 1$, i.e., that $\chi$ is nontrivial, or
equivalently
that $\omega$ is ramified.
Let $\ell = m..n$ be a segment.
If $|m-n| \geq k$,
then $\chi$ induces a character
$K_{\ell} \ni \left(
  \begin{smallmatrix}
    a&b\\
    c&d
  \end{smallmatrix}
\right) \mapsto \chi(d)$;
by abuse of notation, we denote this character also
by $\chi: K_{\ell} \rightarrow \mathbb{C}^\times$.
Define $e_{\ell,\chi} := 0$ unless $|m-n| \geq k$,
in which case set $e_{\ell,\chi}  := \pi(\vol(K_{\ell})^{-1}
\chi^{-1} 1_{K_{\ell}})$.
Denote by $\pi[\ell,\chi]$ the image of $e_{\ell,\chi}$,
thus
$\pi[\ell,\chi] = \{v \in \pi : g v = \chi(g) v \text{ for all } g \in K_{\ell}\}$.
Define $\pi[\ell,\chi]^\sharp$ in terms of
$\pi[\ell,\chi]$ as in \S\ref{sec:main}
and
$e_{\ell,\chi}^*$ in terms of $e_{\ell,\chi}$ as in Theorem \ref{thm:main}.
Local newvector theory gives a formula analogous to \eqref{eqn:dimension-formulas-for-gl2-fixed-vectors}
for $\pi[\ell,\chi]^\sharp$.
The proof of Corollary \ref{cor:for-generics}
shows in this context that
$e_{\ell,\chi}^* = 0$ unless $c(\pi) = \# \ell - 1$,
in which case $e_{\ell,\chi}^*$ is the standard projector
onto the one-dimensional space $\pi[\ell,\chi] =
\pi[\ell,\chi]^\sharp$.
(Alternatively, one could reach the same conclusion
by
running the argument
with $K_{\ell}$ replaced by its subgroup
$\{\left(
  \begin{smallmatrix}
    \ast & \ast \\
    \ast & d
  \end{smallmatrix}
\right)
\in K_\ell : d \in 1 + \mathfrak{p}^k\}$.)
\end{remark}

\section{Deduction of Theorem \ref{cor:concrete} from Theorem \ref{thm:main}}
\label{sec-5}
Theorem \ref{cor:concrete} follows from Theorem \ref{thm:main}
via a standard ``adelic-to-classical'' argument (see
e.g. \cite{MR0379375}); for the sake of completeness and variety
of exposition, we record here a direct ``local-to-classical''  proof of this
implication.  Thus, let $q$ be cubefull.  For squarefree
integers $d,e$ dividing $q$, denote by $E_{d..d e}$ the
orthogonal projection from $\mathcal{A}(q)$ to the subspace
$\{\varphi|_d : \varphi \in \mathcal{A}(\tfrac{q}{d e})\}$.
As in \S\ref{sec:basic-idea}, Theorem \ref{cor:concrete}
amounts to the assertion that the operator
$E_{1..q}^* := \sum_{d,e|q} \mu(d) \mu(e) E_{d..d e}$ defines
the orthogonal projection onto $\mathcal{A}^*(q)$.  For the
same reasons as in the proof of Lemma
\ref{prop:level-p-cubed-projector}, $E_{1..q}^*$ acts by the
identity on the newspace, so it remains only to verify that it
annihilates the oldspace.  The oldspace is the sum over all
$p \mid q$ of the subspaces
\begin{equation}\label{eq:p-old-subspaces}
  \mathcal{A}(q/p) \text{ and }
  \{\varphi|_p : \varphi \in \mathcal{A}(q/p)\}
\end{equation}
so we reduce
to verifying for each such $p$ that $E_{1..q}^*$ annihilates
the spaces \eqref{eq:p-old-subspaces}.  To that
end, denote by $\pi$ the span of the functions $\varphi|\gamma$
taken over all $\varphi \in \mathcal{A}(q)$ and all $\gamma$ in the group $\Gamma := \SL_2(\mathbb{Q}) \cap
R_0(q)[1/p]$,
where $R_0(q) := \left[
  \begin{smallmatrix}
    \mathbb{Z} &\mathbb{Z} \\
    q \mathbb{Z} & \mathbb{Z} 
  \end{smallmatrix}
\right]\leq M_2(\mathbb{Z})$ is the order for which $\SL_2(\mathbb{Q}) \cap R_0(q) = \Gamma_0(q)$.
Regard $\Gamma$ as a subgroup of $G := \SL_2(\mathbb{Q}_p)$.
It is dense.
Each $\varphi \in \pi$ is invariant under some congruence subgroup of $\Gamma_0(q)$, hence under $\Gamma \cap U$ for some open subgroup
$U$ of $G$.
Consequently, the left action of $\Gamma$ on $\pi$
given by $\gamma \cdot \varphi := \varphi | \gamma^{-1}$
extends continuously to a smooth unitary representation
of $G$, which we continue to denote by $\pi$. 
Factor $q = q_0 p^n$ where $(q_0,p) = 1$, and denote by $\ell$
the segment $\ell := 0..n$.   The subspace $\mathcal{A}(q)$ is
recovered from $\pi$ as $\mathcal{A}(q) = \pi[\ell]$,
while $\pi[\ell]^\flat$
is the span of
\eqref{eq:p-old-subspaces}.
By the Chinese
remainder theorem, we may factor
\begin{equation}\label{eq:factor-projcetor}
  E_{1..q}^* = E_{1..q_0}^* \circ e_{\ell}^* |_{\mathcal{A}(q)}
\end{equation}
where
$E_{1..q_0}^*$ is defined analogously to $E_{1..q}^*$ and
$e_{\ell}^*$ is the operator on $\pi$ defined in
\S\ref{sec:main}.
By Theorem \ref{thm:main},
the subspaces \eqref{eq:p-old-subspaces}
are annihilated by
$e_{\ell}^*$,
hence (by \eqref{eq:factor-projcetor})
by $E_{1..q}^*$, as required.

\section{Reduction to an operator calculus for idempotents}
\label{sec-6}
We reduce here the proof of Theorem \ref{thm:main}
to that of the following:
\begin{lemma}\label{lem:composition}
  Let $\ell, \ell'$ be segments.
  Suppose that $\ell \subseteq \ell '$
  or $\ell \supseteq \ell '$ or
  $\# \ell \cap \ell ' \geq 2$.
  Then
  $e_{\ell} \circ e_{\ell'} = e_{\ell \cap \ell '}$.
\end{lemma}
\begin{remark}
  By analogy to composition formulas arising in
  microlocal analysis,
  it may be instructive to
  think of the characteristic function $1_\ell$ of
  $\ell$ as a symbol, $e_\ell$ as its quantization,
  and the conclusion of Lemma \ref{lem:composition}
  as the assertion that the composition of the quantization
  of two such symbols $1_{\ell}, 1_{\ell'}$
  is the quantization of their product
  $1_\ell 1_{\ell'} = 1_{\ell \cap \ell'}$ in nice enough cases.

\end{remark}
\begin{remark}
  Suppose $G = \GL_2(k)$.
  The identity
  $e_{\ell} \circ e_{\ell'} = e_{\ell \cap \ell'}$
  fails in general if $\# \ell \cap \ell ' < 2$ and neither segment contains the other, but continues to hold if $\pi$ is irreducible with unramified central character and log-conductor $c(\pi) \geq 2$, the point being that in such cases, one can simultaneously diagonalize the operators $e_\ell$ by a basis of characteristic functions of $\mathfrak{o}^\times$-cosets in the Kirillov model, corresponding classically to the Fourier coefficients of newforms of conductor divisible by $p^2$ being supported on integers coprime to $p$; see \S\ref{sec:9}.
\end{remark}

Assuming Lemma \ref{lem:composition}
for the moment, we now deduce Theorem \ref{thm:main} by an argument similar to that in the proof of Lemma \ref{prop:level-p-cubed-projector}.
Take $\ell = m..n$ with $|m-n| \geq 3$;
we must show that $e_\ell^*$ is
the standard projector onto $\pi[\ell]^\sharp$.
By the identity
$e_{\ell}^* = e_{\ell}^* \circ e_\ell$
and the definition of $\pi[\ell]^\sharp$,
we see that $e_{\ell}^*$ annihilates the standard complement
of $\pi[\ell]$ and restricts to the identity on
$\pi[\ell]^\sharp$,
so our main task is to show that it annihilates
$\pi[\ell]^\flat$,
or equivalently, that
\begin{equation}\label{eq:comp-kills-oldforms-goal}
  e_\ell^* \circ e_{\ell'} = 0
\end{equation}
for all $\ell ' \subsetneq \ell$.
Since $e_{\ell'} = e_{\ell''} \circ e_{\ell'}$
for any $\ell'' \supseteq \ell'$,
it suffices to establish
\eqref{eq:comp-kills-oldforms-goal}
when $\ell'$ is a \emph{maximal} proper subsegment $\ell'
\subsetneq \ell$.
We verify this when
$\ell' = m+1..n$; the case $\ell' = m..n-1$ is similar.
Our assumption on $|m-n|$ implies that
for each segment
$\ell'' \in \{m..n, m+1..n, m..n-1, m+1..n-1\}$ arising
in the definition of $e_{\ell}^*$,
one has $\# \ell' \cap \ell'' \geq 2$,
so by Lemma \ref{lem:composition},
we have
\[
e_{\ell}^* \circ e_{\ell'}
=
e_{m..n \cap \ell'}
-
e_{m+1..n \cap \ell'}
-
e_{m+1..n-1 \cap \ell'}
-
e_{m+1..n-1 \cap \ell'} \] which simplifies to $e_{m+1..n}
-
e_{m+1..n}
-
e_{m+1..n-1}
-
e_{m+1..n-1}
=0$,
as required.

\section{Group-theoretic proof}
\label{sec-7}
We record here a proof of Lemma \ref{lem:composition}
similar to that of Lemma \ref{prop:level-p-cubed-projector}.
The case in which one of $\ell, \ell '$ contains the other
follows
from the transitivity of standard projectors onto nested
subspaces,
so we need only consider the case that
$\# \ell \cap \ell' \geq 2$
and neither contains the other.
By a symmetry argument,
we reduce further to showing that
\begin{equation}\label{eqn:goal-with-explicit-notation}
  e_{m..n} \circ e_{m'..n'} = e_{m'..n}
  \text{ whenever }
  m < m' < n < n'.
\end{equation}
Note especially
that \eqref{eqn:goal-with-explicit-notation}
implies $|m' - n| \geq 1$.
Since
$e_{m'..n} = e_{m'..n} \circ e_{m'..n'}$,
we reduce
to verifying that
$e_{m..n} v = e_{m'..n} v$ for all $v \in \pi[m'..n']$.
For such $v$, one has
\[
e_{m..n} v = \frac{1}{|K_{m..n}/K_{m..n'}|}
\sum_{\gamma \in K_{m..n}/K_{m..n'}} \pi(\gamma) v,
\]
\[
e_{m'..n} v = \frac{1}{|K_{m'..n}/K_{m'..n'}|}
\sum_{\gamma \in K_{m'..n}/K_{m'..n'}} \pi(\gamma) v.
\]
so our task reduces to verifying that the natural
map of coset spaces
$K_{m..n}/K_{m..n'} \rightarrow K_{m'..n}/K_{m'..n'}$
induced by the inclusions
$K_{m..n} \leq K_{m'..n}$, $K_{m..n'} \leq K_{m'..n'}$
is bijective.
The injectivity follows from 
$K_{m..n} \cap K_{m'..n'} = K_{m..n'}$,
the surjectivity from the claim
$K_{m..n} \cdot  K_{m'..n'} = K_{m'..n}$
for whose proof we must verify that any
\[
\gamma = \begin{pmatrix}
  a & b \\
  c & d
\end{pmatrix}
\in \begin{bmatrix}
  \mathfrak{o}  & \mathfrak{p}^{-m'} \\
  \mathfrak{p}^n & \mathfrak{o} 
\end{bmatrix}^\times \cap G
\]
arises as $\gamma = \gamma_1 \gamma_2$ for some
\[
\gamma_1 \in \begin{bmatrix}
  \mathfrak{o}  &  \mathfrak{p}^{-m} \\
  \mathfrak{p}^n & \mathfrak{o} 
\end{bmatrix}^\times \cap G,
\quad 
\gamma_2 \in \begin{bmatrix}
  \mathfrak{o}  &  \mathfrak{p}^{-m'} \\
  \mathfrak{p}^{n'} & \mathfrak{o} 
\end{bmatrix}^\times \cap G.
\]
From $\det(\gamma) \in \mathfrak{o}^\times$, $b \in
\mathfrak{p}^{-m'}, c \in \mathfrak{p}^n$,
$n > m'$ we obtain $b c \in \mathfrak{p}$ and hence $a, d \in \mathfrak{o}^\times$,
justifying the choice
\[
\gamma_1 :=
\begin{pmatrix}
  1 &  \\
  c/a & 1
\end{pmatrix},
\quad 
\gamma_2 :=
\begin{pmatrix}
  a & b \\
  0 & d - b c / a
\end{pmatrix}
\]
for which $\gamma_1 \in \SL_2(k) < G$
and $\gamma_2 = \gamma_1^{-1} \gamma \in G$.
The required congruences are clear.
This completes the proof.

\section{Probabilistic proof}
\label{sec-8}
We record
here an alternative proof of the key identity
\eqref{eqn:goal-with-explicit-notation}.
For notational simplicity
we suppose that $G = \SL_2(k)$ (cf. Remark \ref{rem:reduce-to-sl2}).
Fix an arbitrary linear functional $v^* : \pi \rightarrow \mathbb{C}$.
It suffices to show that
\begin{equation}\label{eq:goal-after-linear-functional}
  v^*(e_{m..n} e_{m'..n'} v) = v^*(e_{m'..n} v) \text{ for }
  v \in \pi[m..n']
\end{equation}
under the assumptions \eqref{eqn:goal-with-explicit-notation} on the
indices.
Denote by $X$ the $(q+1)$-regular tree, where $q := \# \mathfrak{o}/\mathfrak{p}$,
and by $X^{m..n'}$ the set of non-backtracking paths
$x = (x_m \rightarrow x_{m+1} \rightarrow \dotsb \rightarrow
x_{n'})$ on $X$.
There is a well-known injection
$G/K_{m..n'} \rightarrow X^{m..n'}$
obtained by identifying $X$ with the set of homothety classes $[L]$ of
lattices
$L \subseteq k^2$ and mapping the coset $g K_{m..n'}$ to
the tuple of lattice classes
$([g L_m], [g L_{m+1}], \dotsc, [g L_{n'}])$ where
$L_j := \mathfrak{p}^{-j} \times \mathfrak{o}$ has stabilizer $K_{j..j}$.
Denote by $X_{m..n'}$ the image of $G/K_{m..n'}$ in $X^{m..n'}$
and by $G \ni g \mapsto [g] \in X_{m..n'}$ the induced surjection.
Consider the map
$\phi : \pi[m..n'] \rightarrow \mathbb{C}^{X_{m..n'}}$
from $\pi$ to the space of complex-valued functions on
$X_{m..n'}$
that sends a vector $v \in \pi[m..n']$
to the function $\phi(v) : X_{m..n'} \rightarrow \mathbb{C}$
given by
\[
\phi(v)([g]) := v^*(\pi(g) v),
\]
which is well-defined.
Given
$x, y \in X^{m..n'}$
and a subsegment $\ell = p..p' \subseteq m..n'$,
write $x|_{\ell} = y|_{\ell}$ to denote that
the subpaths $(x_p \rightarrow  \dotsb
\rightarrow x_{p'})$,
$(y_p \rightarrow  \dotsb \rightarrow
y_{p'})$
coincide.
The subset $X_{m..n'}$ of $X^{m..n'}$ has the property that if
$x \in X_{m..n'}$
and $y \in X^{m..n'}$ satisfy $x|_\ell = y|_\ell$ for some
subsegment $\ell \subseteq m..n'$,
then $y \in X_{m..n'}$.
For $f : X_{m..n'} \rightarrow \mathbb{C}$
and a subsegment $\ell \subseteq m..n'$,
denote by
$\rho_{\ell} f: X_{m..n'} \rightarrow \mathbb{C}$
the function given by the expectation
$\rho_{\ell} f(x) = \mathbb{E} f(y)$
taken over $y = (y_m \rightarrow y_{m+1} \rightarrow \dotsb y_{n'})$
chosen uniformly at random
from the finite set of non-backtracking paths for which $x|_{\ell} = y|_{\ell}$.
One verifies directly from the definitions that
\begin{equation} \phi(e_{\ell} v) = \rho_{\ell} \phi(v), \end{equation}
so to establish \eqref{eq:goal-after-linear-functional},
our task reduces to showing for all $f : X_{m..n'} \rightarrow
\mathbb{C}$ that
\begin{equation}\label{eq:goal-paths}
  \rho_{m..n} \rho_{m'..n'} f = \rho_{m'..n} f.
\end{equation}
It suffices to test this equality on the characteristic function
$f := 1_x$ of an arbitrary non-backtracking path $x = (x_m \rightarrow \dotsb \rightarrow
x_{n'}) \in X_{m..n'}$.
The RHS of \eqref{eq:goal-paths} is then the uniform distribution
on the finite set of non-backtracking paths
$y = (y_m \rightarrow \dotsb \rightarrow y_{n'})$
for which $y|_{m'..n} = x|_{m'..n}$,
and so \eqref{eq:goal-paths} follows from:
\begin{lemma}
  Suppose $m \leq m' < n \leq n'$.
  Then the following probability distributions on $X_{m..n'}$ coincide:
  \begin{itemize}
  \item The uniform distribution $\rho_{m'..n} 1_x$ on the finite set
    of non-backtracking paths
    $y = (y_m \rightarrow \dotsb \rightarrow y_{n'})$ for which
    $y|_{m'..n} = x|_{m'..n}$.
  \item The distribution $\rho_{m..n} \rho_{m'..n} 1_x$ generated iteratively as follows:
    \begin{enumerate}
    \item Start with the deterministic subpath $(y_{m'} \rightarrow \dotsb
      \rightarrow y_n) := (x_{m'} \rightarrow \dotsb x_n)$.
    \item Choose uniformly at random a forward extension
      $(y_n \rightarrow \dotsb \rightarrow y_{n'})$
      satisfying the non-backtracking condition
      that $(y_n \rightarrow
      y_{n+1})$ not be the inverse of $(x_n \rightarrow x_{n-1})$.
    \item Independently choose uniformly at random a backward extension
      $(y_{m} \rightarrow \dotsb \rightarrow y_{m'})$
      for which $(y_{m'} \rightarrow
      y_{m' - 1})$ is not the inverse of $(x_{m'} \rightarrow x_{m'+ 1})$.
    \end{enumerate}
  \end{itemize}
\end{lemma}
\begin{proof}
  The assumption $n > m'$ implies that the subpath
  $(x_{m'} \rightarrow \dotsb \rightarrow x_n)$ contains the (possibly identical) edges
  $(x_{m'} \rightarrow x_{m'+1})$
  and $(x_{n-1} \rightarrow x_n)$,
  so the non-backtracking condition
  on the forward path
  $(y_n \rightarrow \dotsb \rightarrow y_{n'})$
  is independent of that on the backward path
  $(y_m \rightarrow \dotsb \rightarrow y_{m'})$.
\end{proof}

\section{A third proof}\label{sec:9}
We sketch here a proof of Corollary \ref{cor:for-generics} (thus
$G = \GL_2(k)$) which is more complicated and less
self-contained than the above proofs, but which some readers may
find illustrative; it is also similar in spirit to what is
typically done classically along the lines discussed in
\S\ref{sec:intro}.  Thus, let notation and assumptions be as in the
statement of Corollary \ref{cor:for-generics}.
We assume, for clarity of exposition,
that $\pi$ is unitary; a proof for general $\pi$ may be obtained by
working systematically with standard projectors in place of
inner products.
We may assume that $\pi$ has unramified central character,
as otherwise both sides of the required identity vanish.
Let us realize $\pi$ in
its Kirillov model with respect to an unramified additive
character (see \cite{Sch02}).

Suppose first that
$c(\pi) \geq 2$, so that $\pi$ is supercuspidal.  Then
$\pi[0..c(\pi)]$ is spanned by the
standard newvector $1_{\mathfrak{o}^\times}$;
more generally,
for any segment $\ell$, one has the decomposition
\[
\pi[\ell] = \oplus _{\substack{
    n \in \mathbb{Z}  : \\
    n..n + c(\pi) \subseteq \ell
  }
}
\mathbb{C}
1_{\varpi^n \mathfrak{o}^\times}
\]
with respect to which the orthogonal projectors
$e_{\ell'}$ for $\ell' \subseteq \ell$
are given by projection onto
the summands
with
$n..n + c(\pi) \subseteq \ell'$.
(Indeed, by the conjugacy of the $K_{\ell}$ for $\ell$ of given
length,
we may assume that $\ell = 0..m$ for some $m \in
\mathbb{Z}_{\geq 0}$;
in that case,
the required conclusion
follows from the proof of the ``absolutely cuspidal'' case of \cite[Thm
1]{MR0337789},
see especially p303-304.)
Lemma \ref{lem:composition}, and hence Corollary
\ref{cor:for-generics}, follows immediately from this
description, even without the assumption
$\# \ell \cap \ell' \geq 2$.

It remains to consider the case
that $c(\pi) \in \{0,1\}$.  We explain the proof when
$c(\pi) = 0$, the case $c(\pi) = 1$ being similar but simpler.
To simplify further, we shall prove the conclusion
$e_{\ell \cap \ell'} = e_{\ell} \circ e_{\ell'}$ of Lemma
\ref{lem:composition} only in the special case
$\ell = 0..2, \ell' = 1..3$ as in \S\ref{sec:basic-idea}; the general case differs
only notationally.  Let $v_0$ be a unit vector in the
one-dimensional space $\pi[0..0]$.  Set
$v_n := \pi(a(\varpi^n)) v_0$, where $a(y) := \diag(y,1)$.  Then
$v_n$ spans $\pi[n..n]$.  Moreover, $v_0,v_1,v_2,v_3$ give a
linear (non-orthogonal) basis for $\pi[0..3]$.  Our task is to verify
that $e_{0..2} e_{1..3} v_i = e_{1..2} v_i$ for
$i=0,1,2,3$.
For $i=1,2$, the vector $v_i$
is preserved
under the indicated idempotents, and the required identity
follows.
It remains to consider the cases $i=0,3$;
they are similar to one another,
so we consider only the case $i=3$.
Since $v_3 = e_{1..3} v_3$,
our task is to show that
\begin{equation}\label{eq:goal-kirillov-computational-approach}
  e_{0..2} v_3 = e_{1..2} v_3.
\end{equation}
For an integer $n$, set
$a_n := \langle v_n, v_0 \rangle$.
Then $a_0 = 1$, $a_{-n} = a_n$ and $\langle v_{m+n}, v_m \rangle = a_n$ for all $m$.
The quantity $a_1$ is the Hecke eigenvalue
normalized so that
$|a_1| \leq 2 / (q^{1/2} + q^{-1/2})$
holds and is sharp for tempered unitary representations.
Since $\pi$ is generic, it is not one-dimensional,
and so $a_1 \neq \pm 1$;
by the determinant test, there exist solutions $b_1, b_2$
to the system of equations
\begin{equation}\label{eq:linear-1}
  a_1 = b_1 a_0 + b_0 a_{-1},
\end{equation}
\begin{equation}\label{eq:linear-2}
  a_2 = b_1 a_1 + b_0 a_0.
\end{equation}
Set $w := b_1 v_2 + b_0 v_1$.
We see from
\eqref{eq:linear-1} that $v_3 - w$ is orthogonal to $v_2$ and
from \eqref{eq:linear-2} that is it orthogonal to $v_1$.  The
$a_n$ are known
(e.g., by the difference equation
for the spherical Whittaker function or the
recurrence relation for the Hecke eigenvalues)
to satisfy a second order linear
difference equation with constants coefficients, so from \eqref{eq:linear-1}
and \eqref{eq:linear-2} we deduce that
\begin{equation}\label{eq:linear-3}
  a_{n+2} = a_{n+1} b_1 + a_n b_0
\end{equation}
for all $n$.
In particular, \eqref{eq:linear-3} holds with $n = 1$.
It follows that $v_3 - w$ is orthogonal also to $v_0$.
Hence $w$ is the orthogonal projection of $v_3$
both to $\langle v_1, v_2 \rangle = \pi[1..2]$
and
to $\langle v_0, v_1, v_2 \rangle = \pi[0..2]$,
giving
the required identity \eqref{eq:goal-kirillov-computational-approach}.

\subsection*{Acknowledgements}
We thank A. Saha and S. Wachter for helpful feedback on an earlier draft
and E. Kowalski for encouragement.
We thank the referee for several helpful comments
which have improved the presentation and clarity of this paper.

\bibliography{refs}{}
\bibliographystyle{plain}
\end{document}